\documentclass[twoside,12pt]{amsart}
\usepackage{amsmath,amssymb,amsthm,longtable,graphicx}

\address{Graduate School of Mathematics Nagoya University, Chikusa-ku 
Nagoya 464-8602 Japan}
\email{kazuto.iijima@math.nagoya-u.ac.jp}
\thanks{}
\dedicatory{}

\newtheorem{definition}{Definition}[section]
\newtheorem{theorem}[definition]{Theorem}
\newtheorem{proposition}[definition]{Proposition}
\newtheorem{example}[definition]{Example}
\newtheorem{corollary}[definition]{Corollary}
\newtheorem{lemma}[definition]{Lemma}
\newtheorem{conjecture}[definition]{Conjecture}

\begin{document}

\title{The first term of plethysms}
\author{K. Iijima}
\date{}
\maketitle

\begin{abstract}
Plethysm of two Schur functions can be expressed as a linear combination of 
Schur functions, and monomial symmetric functions.
In this paper, we express the coefficients combinatorially 
in the case of monomial symmetric functions. 
And by using it, we determine the first term of the plethysm 
with respect to Schur functions under the reverse lexicographic order. 

\end{abstract}

%%%%%%%%%%%%%%%%%%%

\section{Introduction}

Let $\lambda$ and $\mu$ be partitions of positive integers $m$ and $n$, respectively.
The plethysm $s_\lambda[s_\mu]$ is the symmetric function obtained by substituting 
the monomials in $s_\mu$ for the variables of $s_\lambda$.
D.E.Littlewood introduced this operation in 1936[9].

Plethysm of two Schur functions is expressed as a linear combination of Schur functions;
$$ s_\lambda[s_\mu]=\sum_{\nu \vdash mn}a_{\lambda[\mu]}^{\nu}s_\nu \,\,, $$
where $\nu \vdash mn$ means that $\nu$ is a partition of $mn$.

Plethysm appear in some fields, especially representation theory.
For example, the coefficient $a_{\lambda[\mu]}^{\nu}$ is equal to
the multiplicity of the irreducible $GL_N$-module of highest weight $\nu$ 
in a certain $GL_N$-module,
and the multiplicity of the irreducible $S_{mn}$-module type of $\nu$
in a certain $S_n \wr S_m$-module, where $GL_N, S_{mn}$, and $S_n \wr S_m$ are
the general linear group, the symmetric group, and the wreath product of the symmetric group
, respectively.
(See [10;Chapter.1 Appendix A and Appendix B].)
From these interpretations, we see that
each of the coefficient $a_{\lambda[\mu]}^{\nu}$ is a nonnegative integer.

One of most fundamental problems for plethysm is expressing the coefficients
$a_{\lambda[\mu]}^{\nu}$ combinatorially
like Kostka coefficients and Littlewood-Richardson coefficients.
In the case of $\lambda = (2)$ or $(1^2)$, Carr$\acute{e}$ and Leclerc [4] 
found a combinatorial description by using domino tableaux. 
However it is generally open problem.

In [2] Agaoka gave the table of $a_{\lambda[\mu]}^{\nu}$ up to $mn=16$,
and in [11] an another method for calculating $a_{\lambda[\mu]}^{\nu}$ is given.

Moreover plethysm is a necessary tool when we consider some geometric problems.
For example, in [1] Agaoka found a new obstruction of local isometric imbeddings of 
Riemannian submanifolds with codimension 2 by calculating the plethysm $s_3[s_{2,2}]$.

Now we explain our approach briefly.
Plethysm is also expressed as a sum of monomial symmetric functions;
$$ s_\lambda[s_\mu]=\sum_{\nu \vdash mn}Y_{\lambda[\mu]}^{\nu}m_\nu $$
where $m_\nu$ denotes the monomial symmetric function corresponding to partition $\nu$. 
We give a combinatorial description of 
the coefficients $Y_{\lambda[\mu]}^{\nu}$ (see section 3), 
and using this, we determine the {\it first} term of $s_\lambda[s_\mu]$ (see section 4).
(The first term is the maximal element among the partition $\nu$ satisfying 
$a_{\lambda[\mu]}^{\nu} \not= 0$ with respect to the reverse lexicographic order.)

%%%%%%%%%%%%%%%%%%%%%%%

\section{Preliminary}

\subsection{Young tableaux, symmetric functions}

We start with introducing the notations. 
For a positive integer $m$, let $[1,m] = \{ i \in \mathbb{Z} | 1 \le i \le m \}$ be the interval of integers between $1$ and $m$. 
For a positive integer $n$, a {\it partition} of $n$ is 
a non-increasing sequences of non-negative integers summing to $n$.
We write $\lambda \vdash n$ if $\lambda$ is a partition of size $n$.
And we use the same notation $\lambda$ to represent the Young diagram 
corresponding to $\lambda$.
Let $s_\lambda,m_\lambda,$ and $h_\lambda$ denote Schur function, 
monomial symmetric function, and complete symmetric function 
corresponding to $\lambda$, respectively.
Here we use $x_1,x_2,\cdots$ as variables.
And we define a symmetric bilinear form $\langle\,\,,\,\,\rangle$ 
on the ring of symmetric functions 
as follows; $$ \langle s_\lambda,s_\mu \rangle =\delta_{\lambda,\mu} .$$

Next we introduce notations for Young tableaux.
For a given Young diagram $\lambda$, a {\it Young tableau} (of shape $\mu$) is 
a map from the set of cells (in the Young diagram $\lambda$) 
to a totally ordered set $S$. 
For a given Young tableau $T$, the image of $(i,j)$ is denoted by $T(i,j)$ 
and called the $(i,j)$ {\it entry} of $T$. 
A {\it semi standard tableau} is a Young tableau whose entries increase weakly
along the rows and increase strictly down the columns.
For a Young diagram $\lambda$, $\mathrm{\mathrm{SSTab}}(\lambda,S)$ denotes the set of semi standard 
tableaux of shape $\lambda$. 

In particular we can take the set of positive integers as a totally ordered set $S$. 
In this case we write $\mathrm{SSTab}(\lambda , [1,m])$ simply $\mathrm{SSTab}(\lambda)_{\leq m}$. 
For a Young tableau $T$, the {\it weight} of $T$ is the sequence 
$\mathrm{wt}(T)=(\mu_1,\mu_2,\cdots)$, where $\mu_k$ is the number of $T(i,j)$ equal to $k$.
We denote by $\mathrm{SSTab}(\lambda;\mu)$ the set of semi standard tableaux 
of shape $\lambda$ with weight $\mu$. 
For a tableau $T \in \mathrm{SSTab}(\lambda;\mu)$, we define 
$x^{T} = x^{\mathrm{wt}(T)}=x_1^{\mu_1}x_2^{\mu_2}\cdots$.

Next we define a total order in $\mathrm{SSTab}(\lambda)$ which is used in section 3.
For a given semi standard tableau $T$, by reading $T$ from left to right in consecutive
rows, starting from the top to bottom, we obtain the {\it word} $\mathrm{word}(T)$.
We define a total order $>$ on the set of words (in which entry is a positive integer) 
as the lexicographic order.

\begin{definition}
%We define a total order $>$ on $\mathrm{SSTab}(\lambda)$ as follows. 
Let $T,U \in \mathrm{SSTab}(\lambda)$.
We define $T>U$ if $\mathrm{word}(T)>\mathrm{word}(U)$. 
\label{order}
\end{definition}

\begin{example}
Let $T_1=$ %WinTpicVersion3.08
\unitlength 0.1in
\begin{picture}(  6.3500,  4.0000)(  0.4000, -4.6000)
% BOX 2 0 3 0
% 2 85 60 285 460
% 
\special{pn 8}%
\special{pa 86 60}%
\special{pa 286 60}%
\special{pa 286 460}%
\special{pa 86 460}%
\special{pa 86 60}%
\special{fp}%
% BOX 2 0 3 0
% 2 85 60 675 260
% 
\special{pn 8}%
\special{pa 86 60}%
\special{pa 676 60}%
\special{pa 676 260}%
\special{pa 86 260}%
\special{pa 86 60}%
\special{fp}%
% LINE 2 0 3 0
% 2 465 60 465 260
% 
\special{pn 8}%
\special{pa 466 60}%
\special{pa 466 260}%
\special{fp}%
% STR 2 0 3 0
% 3 175 60 175 160 5 0
% $1$
\put(1.7500,-1.6000){\makebox(0,0){$1$}}%
% STR 2 0 3 0
% 3 375 50 375 150 5 0
% $1$
\put(3.7500,-1.5000){\makebox(0,0){$1$}}%
% STR 2 0 3 0
% 3 565 60 565 160 5 0
% $2$
\put(5.6500,-1.6000){\makebox(0,0){$2$}}%
% STR 2 0 3 0
% 3 175 280 175 380 5 0
% $3$
\put(1.7500,-3.8000){\makebox(0,0){$3$}}%
\end{picture}%
, $T_2=$ %WinTpicVersion3.08
\unitlength 0.1in
\begin{picture}(  6.3500,  4.0000)(  0.3000, -4.3000)
% BOX 2 0 3 0
% 2 75 30 275 430
% 
\special{pn 8}%
\special{pa 76 30}%
\special{pa 276 30}%
\special{pa 276 430}%
\special{pa 76 430}%
\special{pa 76 30}%
\special{fp}%
% BOX 2 0 3 0
% 2 75 30 665 230
% 
\special{pn 8}%
\special{pa 76 30}%
\special{pa 666 30}%
\special{pa 666 230}%
\special{pa 76 230}%
\special{pa 76 30}%
\special{fp}%
% LINE 2 0 3 0
% 2 455 30 455 230
% 
\special{pn 8}%
\special{pa 456 30}%
\special{pa 456 230}%
\special{fp}%
% STR 2 0 3 0
% 3 165 30 165 130 5 0
% $1$
\put(1.6500,-1.3000){\makebox(0,0){$1$}}%
% STR 2 0 3 0
% 3 365 20 365 120 5 0
% $1$
\put(3.6500,-1.2000){\makebox(0,0){$1$}}%
% STR 2 0 3 0
% 3 555 30 555 130 5 0
% $2$
\put(5.5500,-1.3000){\makebox(0,0){$2$}}%
% STR 2 0 3 0
% 3 165 250 165 350 5 0
% $4$
\put(1.6500,-3.5000){\makebox(0,0){$4$}}%
\end{picture}%
, $T_3=$ %WinTpicVersion3.08
\unitlength 0.1in
\begin{picture}(  6.3500,  4.0000)(  0.5000, -4.5000)
% BOX 2 0 3 0
% 2 95 50 295 450
% 
\special{pn 8}%
\special{pa 96 50}%
\special{pa 296 50}%
\special{pa 296 450}%
\special{pa 96 450}%
\special{pa 96 50}%
\special{fp}%
% BOX 2 0 3 0
% 2 95 50 685 250
% 
\special{pn 8}%
\special{pa 96 50}%
\special{pa 686 50}%
\special{pa 686 250}%
\special{pa 96 250}%
\special{pa 96 50}%
\special{fp}%
% LINE 2 0 3 0
% 2 475 50 475 250
% 
\special{pn 8}%
\special{pa 476 50}%
\special{pa 476 250}%
\special{fp}%
% STR 2 0 3 0
% 3 185 50 185 150 5 0
% $1$
\put(1.8500,-1.5000){\makebox(0,0){$1$}}%
% STR 2 0 3 0
% 3 385 40 385 140 5 0
% $2$
\put(3.8500,-1.4000){\makebox(0,0){$2$}}%
% STR 2 0 3 0
% 3 575 50 575 150 5 0
% $2$
\put(5.7500,-1.5000){\makebox(0,0){$2$}}%
% STR 2 0 3 0
% 3 185 270 185 370 5 0
% $2$
\put(1.8500,-3.7000){\makebox(0,0){$2$}}%
\end{picture}%
Then $\mathrm{word}(T_1)=1123,\, \mathrm{word}(T_2)=1124,\,\mathrm{word}(T_3)=1222$.
Thus $T_1<T_2<T_3$.
\end{example}

Now we recall well-known results for Kostka coefficients.

\begin{definition}
For $\lambda,\mu \vdash n$, the Kostka coefficient $K_{\lambda,\mu}$ is defined by  
$$ s_\lambda = \sum_{\mu \vdash n}K_{\lambda,\mu}m_\mu.$$
\par Similarly the inverse Kostka coefficient $K_{\lambda,\mu}^{-1}$ is given by 
$$ m_\lambda = \sum_{\mu \vdash n}K_{\lambda,\mu}^{-1}s_\mu.$$
\end{definition}

By a simple consideration, we have the following;

\begin{proposition}
$$ h_\lambda = \sum_{\mu \vdash n}K_{\lambda,\mu}s_\mu.$$
\end{proposition}

Next theorem supply us with a combinatorially expression of Kostka coefficients.

\begin{theorem}
Let $\lambda,\mu \vdash n$, then we have 
$$ K_{\lambda,\mu}= \#\mathrm{SSTab}(\lambda;\mu) .$$
\end{theorem}

$\bf{Remark.}$ 
For $K_{\lambda,\mu}^{-1}$ we also have a combinatorially expression.[5]

\vspace{1em}
From this theorem we have some corollaries which we will use later.

\begin{corollary}
We introduce a total order on the set of Young diagrams 
by the reverse lexicographic order.$($see $[10])$
Then for $\lambda,\mu \vdash n$, 
$$\mathrm{(i)} \, \text{ If } \, \lambda < \mu \,,\, \text{ then } \,\,K_{\lambda,\mu}=0 \,.$$
$$\mathrm{(ii)} \, \text{ If } \,\lambda < \mu \,,\, \text{ then } \,\,K_{\lambda,\mu}^{-1}=0 \,.$$
\end{corollary}

For a positive integer $m$, by putting $x_{m+1}=x_{m+2}=\cdots=0$ in the theorem, 
we have the next corollary.

\begin{corollary}
For a Schur function with $m$ variables $s_\lambda(x_1,\cdots,x_m)$, 
we have $$ s_\lambda(x_1,\cdots,x_m)=\sum_{T \in \mathrm{SSTab}(\mu)_{\leq m}}x^{T}. $$
\end{corollary}

\subsection{plethysm}

Let $f$ and $g$ be two symmetric functions and write $g$ as a sum of monomials:
$g=\sum_{\alpha \in \mathbb{N}^{\infty}} c_{\alpha} x^{\alpha}$. 
Introduce the set of fictitious variables $y_i$ defined by 
$$ \Pi (1+y_i t) = \Pi_{\alpha \in \mathbb{N}^{\infty}} (1+x^{\alpha} t)^{c_{\alpha}} $$
and define $f[g] = f(y_1,y_2,\cdots)$. 
If $f$ is $n$-th symmetric function and $g$ is $m$-th, then $f[g]$ is $nm$-th symmetric function. 
We call this multiple on the set of symmetric functions {\it plethysm}. 

\begin{proposition}
Let $f$ and $g$ be two symmetric functions.
We restrict $g$ to $s$-variables and write it as a sum of monomials:
$$g(x_1,\cdots,x_s,0,0,\cdots)=\sum_{i=1}^N x^{\alpha(i)} .$$ 
Then, 
$$f[g](x_1,\cdots,x_s,0,0,\cdots)= f(x^{\alpha(1)} , \cdots , x^{\alpha(N)} , 0 , 0 ,\cdots ).$$ 
That is, $f[g](x_1,\cdots,x_s,0,0,\cdots)$ is the symmetric polynomial obtained by 
	substituting monomials in $f$ (together with multiplicity) for the variables in $g$. 
\label{prop-plethysm}
\end{proposition}

%%%%%%%%%%%%%%%%%%%%%%%

\section{The expression of plethysm in monomial symmetric functions}

Let $\lambda \vdash m$, $\mu \vdash n$ and $\nu \vdash mn$. 
We put a copy the Young diagram $\mu$ in each cell of the Young diagram $\lambda$, 
and denote such a diagram by $\lambda[\mu]$.
For example, if $\lambda=(3,1)$ and $\mu=(3,2)$, we consider the following diagram.
(Fig.1)

\begin{center}
%WinTpicVersion3.08
\unitlength 0.1in
\begin{picture}( 30.0000, 16.0000)(  6.1000,-18.0000)
% BOX 1 0 3 0
% 2 610 1800 1610 1000
% 
\special{pn 13}%
\special{pa 610 1800}%
\special{pa 1610 1800}%
\special{pa 1610 1000}%
\special{pa 610 1000}%
\special{pa 610 1800}%
\special{fp}%
% BOX 1 0 3 0
% 2 1610 1000 610 200
% 
\special{pn 13}%
\special{pa 1610 1000}%
\special{pa 610 1000}%
\special{pa 610 200}%
\special{pa 1610 200}%
\special{pa 1610 1000}%
\special{fp}%
% BOX 1 0 3 0
% 2 1610 1000 3610 200
% 
\special{pn 13}%
\special{pa 1610 1000}%
\special{pa 3610 1000}%
\special{pa 3610 200}%
\special{pa 1610 200}%
\special{pa 1610 1000}%
\special{fp}%
% BOX 2 0 3 0
% 2 800 800 1200 400
% 
\special{pn 8}%
\special{pa 800 800}%
\special{pa 1200 800}%
\special{pa 1200 400}%
\special{pa 800 400}%
\special{pa 800 800}%
\special{fp}%
% LINE 2 0 3 0
% 8 800 600 1400 600 1400 600 1400 400 1400 400 1200 400 1000 790 1000 390
% 
\special{pn 8}%
\special{pa 800 600}%
\special{pa 1400 600}%
\special{fp}%
\special{pa 1400 600}%
\special{pa 1400 400}%
\special{fp}%
\special{pa 1400 400}%
\special{pa 1200 400}%
\special{fp}%
\special{pa 1000 790}%
\special{pa 1000 390}%
\special{fp}%
% BOX 2 0 3 0
% 2 800 1620 1200 1220
% 
\special{pn 8}%
\special{pa 800 1620}%
\special{pa 1200 1620}%
\special{pa 1200 1220}%
\special{pa 800 1220}%
\special{pa 800 1620}%
\special{fp}%
% LINE 2 0 3 0
% 8 800 1420 1400 1420 1400 1420 1400 1220 1400 1220 1200 1220 1000 1610 1000 1210
% 
\special{pn 8}%
\special{pa 800 1420}%
\special{pa 1400 1420}%
\special{fp}%
\special{pa 1400 1420}%
\special{pa 1400 1220}%
\special{fp}%
\special{pa 1400 1220}%
\special{pa 1200 1220}%
\special{fp}%
\special{pa 1000 1610}%
\special{pa 1000 1210}%
\special{fp}%
% BOX 2 0 3 0
% 2 1810 790 2210 390
% 
\special{pn 8}%
\special{pa 1810 790}%
\special{pa 2210 790}%
\special{pa 2210 390}%
\special{pa 1810 390}%
\special{pa 1810 790}%
\special{fp}%
% LINE 2 0 3 0
% 8 1810 590 2410 590 2410 590 2410 390 2410 390 2210 390 2010 780 2010 380
% 
\special{pn 8}%
\special{pa 1810 590}%
\special{pa 2410 590}%
\special{fp}%
\special{pa 2410 590}%
\special{pa 2410 390}%
\special{fp}%
\special{pa 2410 390}%
\special{pa 2210 390}%
\special{fp}%
\special{pa 2010 780}%
\special{pa 2010 380}%
\special{fp}%
% BOX 2 0 3 0
% 2 2790 810 3190 410
% 
\special{pn 8}%
\special{pa 2790 810}%
\special{pa 3190 810}%
\special{pa 3190 410}%
\special{pa 2790 410}%
\special{pa 2790 810}%
\special{fp}%
% LINE 2 0 3 0
% 8 2790 610 3390 610 3390 610 3390 410 3390 410 3190 410 2990 800 2990 400
% 
\special{pn 8}%
\special{pa 2790 610}%
\special{pa 3390 610}%
\special{fp}%
\special{pa 3390 610}%
\special{pa 3390 410}%
\special{fp}%
\special{pa 3390 410}%
\special{pa 3190 410}%
\special{fp}%
\special{pa 2990 800}%
\special{pa 2990 400}%
\special{fp}%
% LINE 1 0 3 0
% 2 2610 1000 2610 200
% 
\special{pn 13}%
\special{pa 2610 1000}%
\special{pa 2610 200}%
\special{fp}%
\end{picture}%
\par Fig.1
\end{center}

\begin{definition}
A semi standard tableau of shape $\lambda[\mu]$ is a semi standard tableau 
$T:\lambda \rightarrow \mathrm{SSTab}(\mu)$ in the sense of Definition \ref{order}. 
Namely it is filled with $mn$ number of positive integers and
it satisfies following two conditions;
\par $(i)$. Each Young tableau of shape $\mu$ is a semi standard tableau.
\par $(ii)$. These $m$ number of semi standard tableaux form a semi standard tableau
of shape $\lambda$ with respect to the totally order in Definition \ref{order}.

Moreover $\mathrm{SSTab}(\lambda[\mu])$ denotes the set of semi standard tableaux of shape 
$\lambda[\mu]$.
\end{definition}

\begin{definition}
For given $T \in \mathrm{SSTab}(\lambda[\mu])$ we define the weight $\mathrm{wt}(T)$ as usual, 
i.e. $\mathrm{wt}(T)=(\nu_1,\nu_2,\cdots)$, where $\nu_k$ is the number of entries equal to $k$.
For $\lambda \vdash m,\mu \vdash n$ and $\nu \vdash mn$, we put 
$$ Y_{\lambda[\mu]}^{\nu}:=\#\mathrm{SSTab}(\lambda[\mu];\nu) .$$
\end{definition}

\begin{example}
Set $\lambda =(2),\mu =(2)$ and $\nu =(2,1,1)$.
Then the Young tableaux of shape $(2)[(2)]$ with weight $(2,1,1)$ are as follows;

\smallskip
{ \rm %WinTpicVersion3.08
\unitlength 0.1in
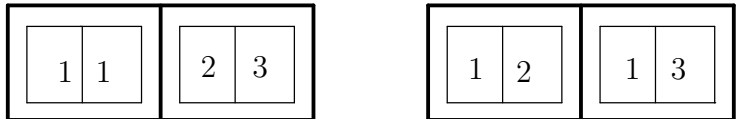
\begin{picture}( 38.0000,  6.0000)(  4.0000, -8.0000)
% LINE 0 0 3 0
% 20 400 200 400 800 400 800 2000 800 2000 800 2000 200 2000 200 400 200 1200 200 1200 800 2600 200 2600 800 2600 800 4200 800 4200 200 4200 800 4200 200 2600 200 3400 200 3400 800
% 
\special{pn 20}%
\special{pa 400 200}%
\special{pa 400 800}%
\special{fp}%
\special{pa 400 800}%
\special{pa 2000 800}%
\special{fp}%
\special{pa 2000 800}%
\special{pa 2000 200}%
\special{fp}%
\special{pa 2000 200}%
\special{pa 400 200}%
\special{fp}%
\special{pa 1200 200}%
\special{pa 1200 800}%
\special{fp}%
\special{pa 2600 200}%
\special{pa 2600 800}%
\special{fp}%
\special{pa 2600 800}%
\special{pa 4200 800}%
\special{fp}%
\special{pa 4200 200}%
\special{pa 4200 800}%
\special{fp}%
\special{pa 4200 200}%
\special{pa 2600 200}%
\special{fp}%
\special{pa 3400 200}%
\special{pa 3400 800}%
\special{fp}%
% STR 2 0 3 0
% 3 660 500 660 600 2 0
% 1
\put(6.6000,-6.0000){\makebox(0,0)[lb]{1}}%
% STR 2 0 3 0
% 3 860 500 860 600 2 0
% 1
\put(8.6000,-6.0000){\makebox(0,0)[lb]{1}}%
% STR 2 0 3 0
% 3 3060 500 3060 600 2 0
% 2
\put(30.6000,-6.0000){\makebox(0,0)[lb]{2}}%
% LINE 2 0 3 0
% 36 500 310 1100 310 1100 310 1100 710 1100 710 500 710 500 710 500 310 1300 310 1900 310 1900 310 1900 710 1900 710 1300 710 1300 710 1300 310 2700 310 3300 310 3300 310 3300 710 3300 710 2700 710 2700 710 2700 310 3500 310 4100 310 4100 710 4100 710 3500 710 3500 310 3500 710 4100 710 4100 710 4100 310 780 310 780 310
% 
\special{pn 8}%
\special{pa 500 310}%
\special{pa 1100 310}%
\special{fp}%
\special{pa 1100 310}%
\special{pa 1100 710}%
\special{fp}%
\special{pa 1100 710}%
\special{pa 500 710}%
\special{fp}%
\special{pa 500 710}%
\special{pa 500 310}%
\special{fp}%
\special{pa 1300 310}%
\special{pa 1900 310}%
\special{fp}%
\special{pa 1900 310}%
\special{pa 1900 710}%
\special{fp}%
\special{pa 1900 710}%
\special{pa 1300 710}%
\special{fp}%
\special{pa 1300 710}%
\special{pa 1300 310}%
\special{fp}%
\special{pa 2700 310}%
\special{pa 3300 310}%
\special{fp}%
\special{pa 3300 310}%
\special{pa 3300 710}%
\special{fp}%
\special{pa 3300 710}%
\special{pa 2700 710}%
\special{fp}%
\special{pa 2700 710}%
\special{pa 2700 310}%
\special{fp}%
\special{pa 3500 310}%
\special{pa 4100 310}%
\special{fp}%
\special{pa 4100 710}%
\special{pa 4100 710}%
\special{fp}%
\special{pa 3500 710}%
\special{pa 3500 310}%
\special{fp}%
\special{pa 3500 710}%
\special{pa 4100 710}%
\special{fp}%
\special{pa 4100 710}%
\special{pa 4100 310}%
\special{fp}%
\special{pa 780 310}%
\special{pa 780 310}%
\special{fp}%
% LINE 2 0 3 0
% 8 790 310 790 710 1590 710 1590 310 2990 310 2990 710 3790 710 3790 310
% 
\special{pn 8}%
\special{pa 790 310}%
\special{pa 790 710}%
\special{fp}%
\special{pa 1590 710}%
\special{pa 1590 310}%
\special{fp}%
\special{pa 2990 310}%
\special{pa 2990 710}%
\special{fp}%
\special{pa 3790 710}%
\special{pa 3790 310}%
\special{fp}%
% STR 2 0 3 0
% 3 3910 430 3910 530 5 0
% 3
\put(39.1000,-5.3000){\makebox(0,0){3}}%
% STR 2 0 3 0
% 3 3670 430 3670 530 5 0
% 1
\put(36.7000,-5.3000){\makebox(0,0){1}}%
% STR 2 0 3 0
% 3 2850 430 2850 530 5 0
% 1
\put(28.5000,-5.3000){\makebox(0,0){1}}%
% STR 2 0 3 0
% 3 1450 420 1450 520 5 0
% 2
\put(14.5000,-5.2000){\makebox(0,0){2}}%
% STR 2 0 3 0
% 3 1720 420 1720 520 5 0
% 3
\put(17.2000,-5.2000){\makebox(0,0){3}}%
\end{picture}%
}

Hence we have $Y_{(2)[(2)]}^{(2,1^2)}=2$ .
\end{example}

\begin{example}
Set $\lambda =(2,1),\mu =(1^2)$ and $\nu =(3,1^3)$.
Then the Young tableaux of shape $(2,1)[(1^2)]$ with weight $(3,1^3)$ are as follows;

\smallskip
{ \rm %WinTpicVersion3.08
\unitlength 0.1in
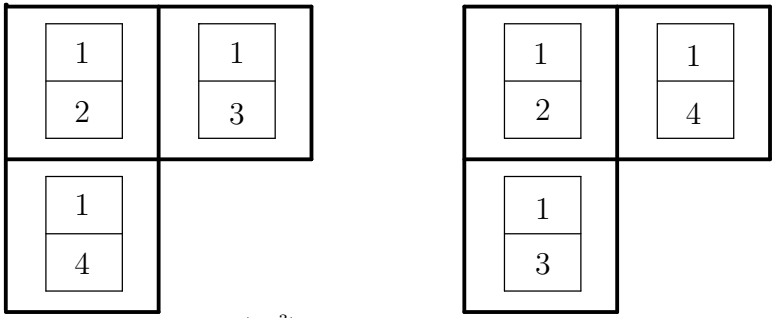
\begin{picture}( 40.0000, 16.1000)( 12.0000,-18.1000)
% LINE 0 0 3 0
% 30 1200 200 1200 200 1200 200 1200 1800 1200 210 2000 210 2000 210 2000 1810 2000 1810 1200 1810 1200 1010 2800 1010 2800 1010 2800 210 2800 210 2000 210 3600 210 3600 1810 3600 1810 4400 1810 4400 1810 4400 210 4400 210 3600 210 4400 210 5200 210 5200 1010 3600 1010 5200 1010 5200 210
% 
\special{pn 20}%
\special{pa 1200 200}%
\special{pa 1200 200}%
\special{fp}%
\special{pa 1200 200}%
\special{pa 1200 1800}%
\special{fp}%
\special{pa 1200 210}%
\special{pa 2000 210}%
\special{fp}%
\special{pa 2000 210}%
\special{pa 2000 1810}%
\special{fp}%
\special{pa 2000 1810}%
\special{pa 1200 1810}%
\special{fp}%
\special{pa 1200 1010}%
\special{pa 2800 1010}%
\special{fp}%
\special{pa 2800 1010}%
\special{pa 2800 210}%
\special{fp}%
\special{pa 2800 210}%
\special{pa 2000 210}%
\special{fp}%
\special{pa 3600 210}%
\special{pa 3600 1810}%
\special{fp}%
\special{pa 3600 1810}%
\special{pa 4400 1810}%
\special{fp}%
\special{pa 4400 1810}%
\special{pa 4400 210}%
\special{fp}%
\special{pa 4400 210}%
\special{pa 3600 210}%
\special{fp}%
\special{pa 4400 210}%
\special{pa 5200 210}%
\special{fp}%
\special{pa 5200 1010}%
\special{pa 3600 1010}%
\special{fp}%
\special{pa 5200 1010}%
\special{pa 5200 210}%
\special{fp}%
% LINE 2 0 3 0
% 50 1410 300 1410 900 1410 900 1810 900 1810 300 1810 900 1810 300 1410 300 2210 300 2610 300 2610 300 2610 900 2610 900 2210 900 2210 900 2210 300 2210 300 2210 300 1410 1100 1810 1100 1810 1700 1810 1100 1410 1100 1410 1700 1410 1700 1810 1700 3810 300 4210 300 4210 300 4210 900 4210 900 3810 900 3810 900 3810 300 3810 1100 4210 1100 4210 1700 4210 1100 4210 1700 3810 1700 3810 1700 3810 1100 4610 900 5010 900 5010 900 5010 300 5010 300 4610 300 4610 300 4610 900
% 
\special{pn 8}%
\special{pa 1410 300}%
\special{pa 1410 900}%
\special{fp}%
\special{pa 1410 900}%
\special{pa 1810 900}%
\special{fp}%
\special{pa 1810 300}%
\special{pa 1810 900}%
\special{fp}%
\special{pa 1810 300}%
\special{pa 1410 300}%
\special{fp}%
\special{pa 2210 300}%
\special{pa 2610 300}%
\special{fp}%
\special{pa 2610 300}%
\special{pa 2610 900}%
\special{fp}%
\special{pa 2610 900}%
\special{pa 2210 900}%
\special{fp}%
\special{pa 2210 900}%
\special{pa 2210 300}%
\special{fp}%
\special{pa 2210 300}%
\special{pa 2210 300}%
\special{fp}%
\special{pa 1410 1100}%
\special{pa 1810 1100}%
\special{fp}%
\special{pa 1810 1700}%
\special{pa 1810 1100}%
\special{fp}%
\special{pa 1410 1100}%
\special{pa 1410 1700}%
\special{fp}%
\special{pa 1410 1700}%
\special{pa 1810 1700}%
\special{fp}%
\special{pa 3810 300}%
\special{pa 4210 300}%
\special{fp}%
\special{pa 4210 300}%
\special{pa 4210 900}%
\special{fp}%
\special{pa 4210 900}%
\special{pa 3810 900}%
\special{fp}%
\special{pa 3810 900}%
\special{pa 3810 300}%
\special{fp}%
\special{pa 3810 1100}%
\special{pa 4210 1100}%
\special{fp}%
\special{pa 4210 1700}%
\special{pa 4210 1100}%
\special{fp}%
\special{pa 4210 1700}%
\special{pa 3810 1700}%
\special{fp}%
\special{pa 3810 1700}%
\special{pa 3810 1100}%
\special{fp}%
\special{pa 4610 900}%
\special{pa 5010 900}%
\special{fp}%
\special{pa 5010 900}%
\special{pa 5010 300}%
\special{fp}%
\special{pa 5010 300}%
\special{pa 4610 300}%
\special{fp}%
\special{pa 4610 300}%
\special{pa 4610 900}%
\special{fp}%
% LINE 2 0 3 0
% 12 1410 600 1810 600 2210 600 2610 600 1810 1400 1410 1400 3810 1400 4210 1400 4210 600 3810 600 4610 600 5010 600
% 
\special{pn 8}%
\special{pa 1410 600}%
\special{pa 1810 600}%
\special{fp}%
\special{pa 2210 600}%
\special{pa 2610 600}%
\special{fp}%
\special{pa 1810 1400}%
\special{pa 1410 1400}%
\special{fp}%
\special{pa 3810 1400}%
\special{pa 4210 1400}%
\special{fp}%
\special{pa 4210 600}%
\special{pa 3810 600}%
\special{fp}%
\special{pa 4610 600}%
\special{pa 5010 600}%
\special{fp}%
% STR 2 0 3 0
% 3 1600 350 1600 450 5 0
% 1
\put(16.0000,-4.5000){\makebox(0,0){1}}%
% STR 2 0 3 0
% 3 1600 1150 1600 1250 5 0
% 1
\put(16.0000,-12.5000){\makebox(0,0){1}}%
% STR 2 0 3 0
% 3 2410 350 2410 450 5 0
% 1
\put(24.1000,-4.5000){\makebox(0,0){1}}%
% STR 2 0 3 0
% 3 4000 360 4000 460 5 0
% 1
\put(40.0000,-4.6000){\makebox(0,0){1}}%
% STR 2 0 3 0
% 3 4010 1170 4010 1270 5 0
% 1
\put(40.1000,-12.7000){\makebox(0,0){1}}%
% STR 2 0 3 0
% 3 4800 370 4800 470 5 0
% 1
\put(48.0000,-4.7000){\makebox(0,0){1}}%
% STR 2 0 3 0
% 3 1600 660 1600 760 5 0
% 2
\put(16.0000,-7.6000){\makebox(0,0){2}}%
% STR 2 0 3 0
% 3 2410 670 2410 770 5 0
% 3
\put(24.1000,-7.7000){\makebox(0,0){3}}%
% STR 2 0 3 0
% 3 1600 1460 1600 1560 5 0
% 4
\put(16.0000,-15.6000){\makebox(0,0){4}}%
% STR 2 0 3 0
% 3 4010 650 4010 750 5 0
% 2
\put(40.1000,-7.5000){\makebox(0,0){2}}%
% STR 2 0 3 0
% 3 4010 1460 4010 1560 5 0
% 3
\put(40.1000,-15.6000){\makebox(0,0){3}}%
% STR 2 0 3 0
% 3 4800 670 4800 770 5 0
% 4
\put(48.0000,-7.7000){\makebox(0,0){4}}%
\end{picture}%
 }

Hence we have $Y_{(2,1)[(1^2)]}^{(3,1^3)}=2$ .
\end{example}

Now we prove the first main result in this paper.

\begin{theorem}
Let $\lambda \vdash m,\mu \vdash n$ and $\nu \vdash mn$.
Then $Y_{\lambda[\mu]}^{\nu}$ is equal to the coefficient of $m_\nu$ in the expansion of 
$s_\lambda[s_\mu]$ in terms of monomial symmetric functions.
\par In other words, we have 
$$ s_\lambda[s_\mu]= \sum_{\nu \vdash mn}Y_{\lambda[\mu]}^{\nu} \, m_\nu .$$
\end{theorem}

\begin{proof}

Before the proof, we introduce some notations.
For a positive integer $s$, we set $r=\#\mathrm{SSTab}(\mu)_{\le s}$. 
For $1 \le i \le r$, let $T_i$ be the $i$-th largest semi standard tableau in $\mathrm{SSTab}(\mu)_{\le s}$ and 
set $y_i = x^{T_i}$. 
In particular, $ \mathrm{SSTab}(\mu)_{\le s}=\{ T_1,\cdots,T_r \} $. 
Note that there is a natural bijection 
$$
\begin{array}{ccccc}
\iota \colon  
	\mathrm{SSTab}(\lambda)_{\le r} & \longrightarrow 
	& \mathrm{SSTab}(\lambda , \mathrm{SSTab}(\mu)_{\le s}) & \longrightarrow 
	& \mathrm{SSTab}(\lambda [\mu])_{\le s} 
\end{array}
$$
such that $y^{U} = x^{\iota(U)}$ for $U \in \mathrm{SSTab}(\lambda)_{\le r}$. 
For example, if $\lambda = (1^2)$ , $\mu = (2)$, $T_1= $ %WinTpicVersion3.08
\unitlength 0.1in
\begin{picture}(  4.4500,  2.0000)(  0.5000, -2.4000)
% BOX 2 0 3 0
% 2 85 40 495 240
% 
\special{pn 8}%
\special{pa 86 40}%
\special{pa 496 40}%
\special{pa 496 240}%
\special{pa 86 240}%
\special{pa 86 40}%
\special{fp}%
% LINE 2 0 3 0
% 2 285 40 285 240
% 
\special{pn 8}%
\special{pa 286 40}%
\special{pa 286 240}%
\special{fp}%
% STR 2 0 3 0
% 3 185 40 185 140 5 0
% $1$
\put(1.8500,-1.4000){\makebox(0,0){$1$}}%
% STR 2 0 3 0
% 3 395 50 395 150 5 0
% $1$
\put(3.9500,-1.5000){\makebox(0,0){$1$}}%
\end{picture}%
 , $T_2 = $ %WinTpicVersion3.08
\unitlength 0.1in
\begin{picture}(  4.4500,  2.0000)(  0.5000, -2.4000)
% BOX 2 0 3 0
% 2 85 40 495 240
% 
\special{pn 8}%
\special{pa 86 40}%
\special{pa 496 40}%
\special{pa 496 240}%
\special{pa 86 240}%
\special{pa 86 40}%
\special{fp}%
% LINE 2 0 3 0
% 2 285 40 285 240
% 
\special{pn 8}%
\special{pa 286 40}%
\special{pa 286 240}%
\special{fp}%
% STR 2 0 3 0
% 3 185 40 185 140 5 0
% $1$
\put(1.8500,-1.4000){\makebox(0,0){$1$}}%
% STR 2 0 3 0
% 3 395 50 395 150 5 0
% $2$
\put(3.9500,-1.5000){\makebox(0,0){$2$}}%
\end{picture}%
 and 
	$U=$ %WinTpicVersion3.08
\unitlength 0.1in
\begin{picture}(  4.1000,  4.1000)(  0.8000, -4.7000)
% BOX 0 0 3 0
% 2 80 60 490 470
% 
\special{pn 20}%
\special{pa 80 60}%
\special{pa 490 60}%
\special{pa 490 470}%
\special{pa 80 470}%
\special{pa 80 60}%
\special{fp}%
% LINE 0 0 3 0
% 2 90 260 480 260
% 
\special{pn 20}%
\special{pa 90 260}%
\special{pa 480 260}%
\special{fp}%
% STR 2 0 3 0
% 3 270 60 270 160 5 0
% $1$
\put(2.7000,-1.6000){\makebox(0,0){$1$}}%
% STR 2 0 3 0
% 3 280 280 280 380 5 0
% $2$
\put(2.8000,-3.8000){\makebox(0,0){$2$}}%
\end{picture}%
 , then 
$\iota(U)=$ %WinTpicVersion3.08
\unitlength 0.1in
\begin{picture}(  6.0000,  6.3000)(  0.9000, -6.9000)
% BOX 2 0 3 0
% 2 175 440 585 640
% 
\special{pn 8}%
\special{pa 176 440}%
\special{pa 586 440}%
\special{pa 586 640}%
\special{pa 176 640}%
\special{pa 176 440}%
\special{fp}%
% LINE 2 0 3 0
% 2 375 440 375 640
% 
\special{pn 8}%
\special{pa 376 440}%
\special{pa 376 640}%
\special{fp}%
% STR 2 0 3 0
% 3 275 440 275 540 5 0
% $1$
\put(2.7500,-5.4000){\makebox(0,0){$1$}}%
% STR 2 0 3 0
% 3 485 450 485 550 5 0
% $2$
\put(4.8500,-5.5000){\makebox(0,0){$2$}}%
% BOX 2 0 3 0
% 2 175 110 585 310
% 
\special{pn 8}%
\special{pa 176 110}%
\special{pa 586 110}%
\special{pa 586 310}%
\special{pa 176 310}%
\special{pa 176 110}%
\special{fp}%
% LINE 2 0 3 0
% 2 375 110 375 310
% 
\special{pn 8}%
\special{pa 376 110}%
\special{pa 376 310}%
\special{fp}%
% STR 2 0 3 0
% 3 275 110 275 210 5 0
% $1$
\put(2.7500,-2.1000){\makebox(0,0){$1$}}%
% STR 2 0 3 0
% 3 485 120 485 220 5 0
% $1$
\put(4.8500,-2.2000){\makebox(0,0){$1$}}%
% BOX 1 0 3 0
% 2 90 60 690 690
% 
\special{pn 13}%
\special{pa 90 60}%
\special{pa 690 60}%
\special{pa 690 690}%
\special{pa 90 690}%
\special{pa 90 60}%
\special{fp}%
% LINE 1 0 3 0
% 2 90 380 690 390
% 
\special{pn 13}%
\special{pa 90 380}%
\special{pa 690 390}%
\special{fp}%
\end{picture}%
 , 
	$y_1=x^{T_1}=x_1^2$, $y_2=x^{T_2}=x_1x_2$, $y^U=y_1y_2$ and $x^{\iota(U)}=x_1^3x_2$.

By Corollary 2.7, in the case of $s$-variables we have; 
$$ 
s_\mu(x_1,\cdots,x_s)
=\sum_{T \in \mathrm{SSTab}(\mu)_{\le s}}x^{T} 
=y_1 + y_2 + \cdots + y_r.
$$
Thus by Proposition \ref{prop-plethysm}, we have;
\begin{align*}
s_\lambda[s_\mu](x_1,\cdots,x_s) 
&= s_\lambda(x^{T_1},\cdots,x^{T_r})  = s_\lambda(y_1,\cdots,y_r) \\
&= \sum_{U \in \mathrm{SSTab}(\lambda)_{\le r}} y^{U} \\
&= \sum_{\iota(U) \in \mathrm{SSTab}(\lambda [\mu])_{\le s}} x^{\iota(U)} .
\end{align*}
Here by taking the limit $s \rightarrow \infty$, 
we have the following equality as symmetric function;
$$ s_\lambda[s_\mu]
=\sum_{T \in \mathrm{SSTab}(\lambda [\mu])} x^T 
=\sum_{\nu \vdash mn}Y_{\lambda[\mu]}^{\nu} \, m_\nu .
$$
\end{proof}

%%%%%%%%%%%%%%%%%%%%%%

\section{The first term of plethysm}

\begin{definition}
Let $\lambda \vdash m$ and $\mu \vdash n$.
The {\it first term} of the plethysm $s_\lambda[s_\mu]$ is the maximal element in 
the set $\{ \nu \vdash mn | a_{\lambda[\mu]}^{\nu} \not= 0 \}$ with respect to 
the reverse lexicographic order.  

\end{definition}

\begin{theorem} 
{\rm([2;Conjecture 2],[3;Conjecture 1.2] and [12;Conjecture 5.1])}
\par Let $\lambda \vdash m$,$\mu \vdash n$,$l=l(\lambda)$ and $l'=l(\mu)$.
$($Where $l=l(\lambda)$ is the length of partition $\lambda$.$)$
Then the first term of plethysm $s_\lambda[s_\mu]$ is 
$$ \nu_0:=(m\mu_1,m\mu_2,\cdots,m(\mu_{l'}-1)+\lambda_1,\lambda_2,\cdots,\lambda_l). $$
Moreover the coefficient of the first term is equal to $1$.

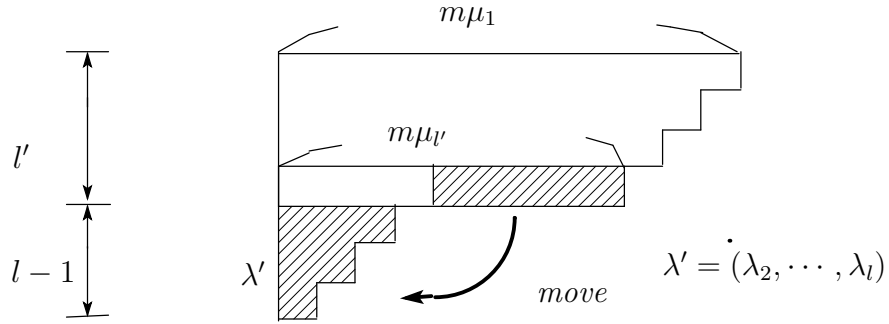
\begin{center}
%WinTpicVersion3.08
\unitlength 0.1in
\begin{picture}( 38.2000, 17.1000)(  2.0000,-18.1000)
% VECTOR 2 0 3 0
% 4 600 1430 600 1210 600 1420 600 1780
% 
\special{pn 8}%
\special{pa 600 1430}%
\special{pa 600 1210}%
\special{fp}%
\special{sh 1}%
\special{pa 600 1210}%
\special{pa 580 1278}%
\special{pa 600 1264}%
\special{pa 620 1278}%
\special{pa 600 1210}%
\special{fp}%
\special{pa 600 1420}%
\special{pa 600 1780}%
\special{fp}%
\special{sh 1}%
\special{pa 600 1780}%
\special{pa 620 1714}%
\special{pa 600 1728}%
\special{pa 580 1714}%
\special{pa 600 1780}%
\special{fp}%
% VECTOR 2 0 3 0
% 4 600 690 600 410 600 640 600 1170
% 
\special{pn 8}%
\special{pa 600 690}%
\special{pa 600 410}%
\special{fp}%
\special{sh 1}%
\special{pa 600 410}%
\special{pa 580 478}%
\special{pa 600 464}%
\special{pa 620 478}%
\special{pa 600 410}%
\special{fp}%
\special{pa 600 640}%
\special{pa 600 1170}%
\special{fp}%
\special{sh 1}%
\special{pa 600 1170}%
\special{pa 620 1104}%
\special{pa 600 1118}%
\special{pa 580 1104}%
\special{pa 600 1170}%
\special{fp}%
% LINE 2 0 3 0
% 6 490 400 710 400 490 1200 720 1200 490 1790 710 1780
% 
\special{pn 8}%
\special{pa 490 400}%
\special{pa 710 400}%
\special{fp}%
\special{pa 490 1200}%
\special{pa 720 1200}%
\special{fp}%
\special{pa 490 1790}%
\special{pa 710 1780}%
\special{fp}%
% STR 2 0 3 0
% 3 200 1520 200 1620 2 0
% $l-1$
\put(2.0000,-16.2000){\makebox(0,0)[lb]{$l-1$}}%
% STR 2 0 3 0
% 3 210 910 210 1010 2 0
% $l'$
\put(2.1000,-10.1000){\makebox(0,0)[lb]{$l'$}}%
% LINE 2 0 3 0
% 2 1600 400 1600 1810
% 
\special{pn 8}%
\special{pa 1600 400}%
\special{pa 1600 1810}%
\special{fp}%
% LINE 2 0 3 0
% 12 1600 1800 1800 1800 1800 1790 1800 1610 1800 1610 2000 1610 2000 1610 2000 1400 2000 1400 2210 1400 2210 1390 2210 1200
% 
\special{pn 8}%
\special{pa 1600 1800}%
\special{pa 1800 1800}%
\special{fp}%
\special{pa 1800 1790}%
\special{pa 1800 1610}%
\special{fp}%
\special{pa 1800 1610}%
\special{pa 2000 1610}%
\special{fp}%
\special{pa 2000 1610}%
\special{pa 2000 1400}%
\special{fp}%
\special{pa 2000 1400}%
\special{pa 2210 1400}%
\special{fp}%
\special{pa 2210 1390}%
\special{pa 2210 1200}%
\special{fp}%
% BOX 2 0 3 0
% 2 1600 1000 3410 1210
% 
\special{pn 8}%
\special{pa 1600 1000}%
\special{pa 3410 1000}%
\special{pa 3410 1210}%
\special{pa 1600 1210}%
\special{pa 1600 1000}%
\special{fp}%
% LINE 2 0 3 0
% 12 3410 1000 3610 1000 3610 990 3610 810 3610 810 3810 810 3810 810 3810 600 3810 600 4020 600 4020 590 4020 400
% 
\special{pn 8}%
\special{pa 3410 1000}%
\special{pa 3610 1000}%
\special{fp}%
\special{pa 3610 990}%
\special{pa 3610 810}%
\special{fp}%
\special{pa 3610 810}%
\special{pa 3810 810}%
\special{fp}%
\special{pa 3810 810}%
\special{pa 3810 600}%
\special{fp}%
\special{pa 3810 600}%
\special{pa 4020 600}%
\special{fp}%
\special{pa 4020 590}%
\special{pa 4020 400}%
\special{fp}%
% LINE 2 0 3 0
% 4 1600 410 4010 410 2410 990 2410 1200
% 
\special{pn 8}%
\special{pa 1600 410}%
\special{pa 4010 410}%
\special{fp}%
\special{pa 2410 990}%
\special{pa 2410 1200}%
\special{fp}%
% LINE 3 0 3 0
% 38 2900 1000 2690 1210 2840 1000 2630 1210 2780 1000 2570 1210 2720 1000 2510 1210 2660 1000 2450 1210 2600 1000 2410 1190 2540 1000 2410 1130 2480 1000 2410 1070 2960 1000 2750 1210 3020 1000 2810 1210 3080 1000 2870 1210 3140 1000 2930 1210 3200 1000 2990 1210 3260 1000 3050 1210 3320 1000 3110 1210 3380 1000 3170 1210 3410 1030 3230 1210 3410 1090 3290 1210 3410 1150 3350 1210
% 
\special{pn 4}%
\special{pa 2900 1000}%
\special{pa 2690 1210}%
\special{fp}%
\special{pa 2840 1000}%
\special{pa 2630 1210}%
\special{fp}%
\special{pa 2780 1000}%
\special{pa 2570 1210}%
\special{fp}%
\special{pa 2720 1000}%
\special{pa 2510 1210}%
\special{fp}%
\special{pa 2660 1000}%
\special{pa 2450 1210}%
\special{fp}%
\special{pa 2600 1000}%
\special{pa 2410 1190}%
\special{fp}%
\special{pa 2540 1000}%
\special{pa 2410 1130}%
\special{fp}%
\special{pa 2480 1000}%
\special{pa 2410 1070}%
\special{fp}%
\special{pa 2960 1000}%
\special{pa 2750 1210}%
\special{fp}%
\special{pa 3020 1000}%
\special{pa 2810 1210}%
\special{fp}%
\special{pa 3080 1000}%
\special{pa 2870 1210}%
\special{fp}%
\special{pa 3140 1000}%
\special{pa 2930 1210}%
\special{fp}%
\special{pa 3200 1000}%
\special{pa 2990 1210}%
\special{fp}%
\special{pa 3260 1000}%
\special{pa 3050 1210}%
\special{fp}%
\special{pa 3320 1000}%
\special{pa 3110 1210}%
\special{fp}%
\special{pa 3380 1000}%
\special{pa 3170 1210}%
\special{fp}%
\special{pa 3410 1030}%
\special{pa 3230 1210}%
\special{fp}%
\special{pa 3410 1090}%
\special{pa 3290 1210}%
\special{fp}%
\special{pa 3410 1150}%
\special{pa 3350 1210}%
\special{fp}%
% LINE 3 0 3 0
% 36 2000 1420 1810 1610 2150 1210 1600 1760 1800 1620 1620 1800 1800 1680 1680 1800 1800 1740 1740 1800 2090 1210 1600 1700 2030 1210 1600 1640 1970 1210 1600 1580 1910 1210 1600 1520 1850 1210 1600 1460 1790 1210 1600 1400 1730 1210 1600 1340 1670 1210 1600 1280 2200 1220 2020 1400 2210 1270 2080 1400 2210 1330 2140 1400 2000 1480 1870 1610 2000 1540 1930 1610
% 
\special{pn 4}%
\special{pa 2000 1420}%
\special{pa 1810 1610}%
\special{fp}%
\special{pa 2150 1210}%
\special{pa 1600 1760}%
\special{fp}%
\special{pa 1800 1620}%
\special{pa 1620 1800}%
\special{fp}%
\special{pa 1800 1680}%
\special{pa 1680 1800}%
\special{fp}%
\special{pa 1800 1740}%
\special{pa 1740 1800}%
\special{fp}%
\special{pa 2090 1210}%
\special{pa 1600 1700}%
\special{fp}%
\special{pa 2030 1210}%
\special{pa 1600 1640}%
\special{fp}%
\special{pa 1970 1210}%
\special{pa 1600 1580}%
\special{fp}%
\special{pa 1910 1210}%
\special{pa 1600 1520}%
\special{fp}%
\special{pa 1850 1210}%
\special{pa 1600 1460}%
\special{fp}%
\special{pa 1790 1210}%
\special{pa 1600 1400}%
\special{fp}%
\special{pa 1730 1210}%
\special{pa 1600 1340}%
\special{fp}%
\special{pa 1670 1210}%
\special{pa 1600 1280}%
\special{fp}%
\special{pa 2200 1220}%
\special{pa 2020 1400}%
\special{fp}%
\special{pa 2210 1270}%
\special{pa 2080 1400}%
\special{fp}%
\special{pa 2210 1330}%
\special{pa 2140 1400}%
\special{fp}%
\special{pa 2000 1480}%
\special{pa 1870 1610}%
\special{fp}%
\special{pa 2000 1540}%
\special{pa 1930 1610}%
\special{fp}%
% STR 2 0 3 0
% 3 1000 1520 1000 1620 2 0
% $\lambda'$
\put(14.0000,-16.2000){\makebox(0,0)[lb]{$\lambda'$}}%
% CIRCLE 0 0 3 0
% 4 2420 1270 2250 1650 2420 1710 2940 1270
% 
\special{pn 20}%
\special{ar 2420 1270 416 416  6.2831853 6.2831853}%
\special{ar 2420 1270 416 416  0.0000000 1.5707963}%
% VECTOR 0 0 3 0
% 4 2400 1680 2270 1690 3960 1390 3960 1390
% 
\special{pn 20}%
\special{pa 2400 1680}%
\special{pa 2270 1690}%
\special{fp}%
\special{sh 1}%
\special{pa 2270 1690}%
\special{pa 2338 1706}%
\special{pa 2324 1686}%
\special{pa 2336 1666}%
\special{pa 2270 1690}%
\special{fp}%
\special{pa 3960 1390}%
\special{pa 3960 1390}%
\special{fp}%
% STR 2 0 3 0
% 3 2950 1610 2950 1710 2 0
% move
\put(29.5000,-17.1000){\makebox(0,0)[lb]{move}}%
% STR 2 0 3 0
% 3 3610 1510 3610 1610 2 0
% $\lambda' = (\lambda_2 , \cdots , \lambda_l)$
\put(36.1000,-16.1000){\makebox(0,0)[lb]{$\lambda' = (\lambda_2 , \cdots , \lambda_l)$}}%
% STR 2 0 3 0
% 3 2440 170 2440 270 2 0
% $m \mu_1$
\put(24.4000,-2.7000){\makebox(0,0)[lb]{$m \mu_1$}}%
% LINE 2 0 3 0
% 10 1600 400 1760 310 1760 310 1910 270 4000 400 3810 300 3820 300 3650 270 3650 270 3580 260
% 
\special{pn 8}%
\special{pa 1600 400}%
\special{pa 1760 310}%
\special{fp}%
\special{pa 1760 310}%
\special{pa 1910 270}%
\special{fp}%
\special{pa 4000 400}%
\special{pa 3810 300}%
\special{fp}%
\special{pa 3820 300}%
\special{pa 3650 270}%
\special{fp}%
\special{pa 3650 270}%
\special{pa 3580 260}%
\special{fp}%
% STR 2 0 3 0
% 3 2170 800 2170 900 2 0
% $m \mu_l'$
\put(21.7000,-9.0000){\makebox(0,0)[lb]{$m \mu_{l'}$}}%
% LINE 2 0 3 0
% 8 1600 1000 1760 930 1760 920 1930 890 3410 1010 3350 890 3350 890 3210 850
% 
\special{pn 8}%
\special{pa 1600 1000}%
\special{pa 1760 930}%
\special{fp}%
\special{pa 1760 920}%
\special{pa 1930 890}%
\special{fp}%
\special{pa 3410 1010}%
\special{pa 3350 890}%
\special{fp}%
\special{pa 3350 890}%
\special{pa 3210 850}%
\special{fp}%
\end{picture}%
\end{center}

\end{theorem}

\begin{proof}
By proposition 2.3, 
note that $$ s_\nu = \sum_{\kappa}K_{\kappa,\nu}^{-1} \,\, h_\kappa. $$
Then we have;
\begin{align*}
 a_{\lambda[\mu]}^{\nu} &= \langle s_\lambda[s_\mu] , s_\nu \rangle \\
  &= \langle s_\lambda[s_\mu] , \sum_{\kappa}K_{\kappa,\nu}^{-1} \,\, h_\kappa \rangle \\
  &= \sum_{\kappa} K_{\kappa,\nu}^{-1}\,\, \langle s_\lambda[s_\mu] , h_\kappa \rangle \\
  &= \sum_{\kappa} K_{\kappa,\nu}^{-1} \,\, Y_{\lambda[\mu]}^{\kappa}
   \,\,\,,\,\,\,(\text{ by Theorem 3.5 and property of }
    \langle \,\,,\,\, \rangle )  \\
  &= Y_{\lambda,\mu}^{\nu} \,+\, \sum_{\kappa > \nu}
   K_{\kappa,\nu}^{-1} \,\, Y_{\lambda[\mu]}^{\kappa} \,\,\,,\,\,\,
   (\text{by Corollary 2.6 (ii)})
\end{align*}

Thus the assertion follows from the next lemma.

\begin{lemma}
\begin{itemize}
\item[(1)] 
	$\mathrm{max} \,\{\,\nu \vdash mn \,|\, Y_{\lambda[\mu]}^{\nu}\not=0\,\} =\nu_0$. 
\item[(2)] 
	$Y_{\lambda[\mu]}^{\nu_0} =1$.  
\end{itemize}
\end{lemma}

\begin{proof}

We define a total order on the set of monomials in $x_i$'s $(i=1,2,\cdots)$ 
by lexicographic order.
(For example, $x_2^3<x_1x_2<x_1^2x_2<x_1^3$.)
Then we can arrange elements of $\mathrm{SSTab}(\mu)$ according to this order as follows;

\begin{center}
\input{figure3.tex}
\end{center}

And weight of these are 
$$ \begin{cases}
\mathrm{wt}(T_1) &=( \mu_1,\mu_2,\cdots,\mu_{l'}) \\
\mathrm{wt}(T_2) &=(\mu_1,\mu_2,\cdots,\mu_{l'}-1,1) \\
\mathrm{wt}(T_3) &=(\mu_1,\mu_2,\cdots,\mu_{l'}-1,0,1) \\
\,\,\, \vdots   & \hspace{4em} \vdots \\
\mathrm{wt}(T_l) &=(\mu_1,\mu_2,\cdots,\mu_{l'}-1,0,\cdots,0,1) .
\end{cases}  $$

Thus the sequence $T^{(1)} \ge T^{(2)} \ge \cdots \ge T^{(m)}$ that have 
a maximal weight under the condition $Y_{\lambda[\mu]}^{\nu} \not=0$ is only 
$$ \begin{cases}
T^{(1)}=\cdots =T^{(\lambda_1)} &=T_1 \\
T^{(\lambda_1+1)}=\cdots =T^{(\lambda_1+\lambda_2)} &=T_2 \\
\,\,\, \vdots   \phantom{abcdefghilklmno} \vdots & \phantom{abc} \vdots \\
T^{(m-\lambda_l+1)}=\cdots =T^{(m)} &=T_l \,\,\,\,\,\,.
\end{cases}  $$

Therefore the maximal weight is 
\begin{align*}
\mathrm{wt}(T^{(1)})+\cdots+\mathrm{wt}(T^{(m)}) &= \lambda_1\mathrm{wt}(T_1)+\cdots +\lambda_l \mathrm{wt}(T_l) \\
  &= (m \mu_1,m \mu_2,m (\mu_{l'}-1)+\lambda_1,\lambda_2,\cdots,\lambda_l) \\
  &= \nu_0
\end{align*}
\end{proof}
\end{proof}

In particular, from the first half of this proof we get a new combinatorial formula 
for plethysm.
$($Note that we can also express $K_{\kappa,\nu}^{-1}$ combinatorially.[6] $)$

\begin{corollary}
$$ a_{\lambda[\mu]}^{\nu}=
  \sum_{\kappa \vdash mn}K_{\kappa,\nu}^{-1}\,\,Y_{\lambda[\mu]}^{\kappa} .$$
\end{corollary}

\begin{example}
We calculate $a_{(2)[(2)]}^{(2,2)}$ from this formula.
Since $K_{(2,2),(2,2)}^{-1}=1 , K_{(3,1),(2,2)}^{-1}=-1 , K_{(4),(2,2)}^{-1}=0 ,
 Y_{(2)[(2)]}^{(2,2)}=2 $ and $ Y_{(2)[(2)]}^{(3,1)}=1 $, we have
$$a_{(2)[(2)]}^{(2,2)}=K_{(2,2),(2,2)}^{-1}\,Y_{(2)[(2)]}^{(2,2)}+
  K_{(3,1),(2,2)}^{-1}\,Y_{(2)[(2)]}^{(3,1)}+K_{(4),(2,2)}^{-1}\,Y_{(2)[(2)]}^{(4)}
  =2-1+0=1.$$
\end{example}

%%%%%%%%%%%%%%%%%%%%%

\section{Some remarks}

Our purpose for plethysm is stated as follows.

$\bf{Problem.}$ 
Express the expansion coefficients $a_{\lambda[\mu]}^{\nu}$ combinatorially.

\vspace{1em}
For this problem, by imitating a proof of Littlewood-Richardson rule given in [7], 
we have the followings.

\vspace{1em}
Let $l=l(\nu)$.
By Jacobi-Trudi's formula, we have
$$ s_\nu=\sum_{\pi \in S_l}\mathrm{sgn}(\pi)h_{\pi *\nu} \phantom{XX},
\phantom{X} (\pi *\nu=(\nu_{\pi(i)}-\pi_i+i)_{1 \le i \le l}) \,\,\,.$$
Here by Proposition 3.5, we have

\begin{align*}
a_{\lambda[\mu]}^{\nu} &= \langle s_\lambda[s_\mu],s_\nu \rangle  \\
  &= \sum_{\pi \in S_l}\mathrm{sgn}(\pi) \langle s_\lambda[s_\mu],h_{\pi *\nu} \rangle   \\
  &= \sum_{\pi \in S_l}\mathrm{sgn}(\pi)\,\, Y_{\lambda[\mu]}^{\pi*\nu} \,\,\,.
\end{align*}

So set $ A=\{ (\pi,T)|\,\pi \in S_l \,,\, T \in \mathrm{SSTab}(\lambda[\mu];\pi * \nu)\, \} $, 
then we have $$ a_{\lambda[\mu]}^{\nu}=\sum_{(\pi,T)\in A }\mathrm{sgn}(\pi)\,. $$
Therefore the following conjecture is expected.

\begin{conjecture}
There are a subset $S_0 \subset \mathrm{SSTab}(\lambda[\mu])$ and 
a bijective map 
$\phi \colon A-A_0 \ni (\pi,T)\rightarrow (\pi ',T') \in A-A_0$ 
such that $\mathrm{sgn}(\pi)=-\mathrm{sgn}(\pi ')$, where $A_0=\{ (\pi,T) \in A|\, T \in S_0 \, \} $ .
\end{conjecture}

In the case of Littlewood-Richardson rule, it is possible to take 
"the set of lattice permutations" as $A_0$.
Then $\phi$ can be defined "{\it properly}".$([7])$

\vspace{1em}
Indeed, the following property holds.

\begin{lemma}
For any $\lambda \vdash m,\mu \vdash n$ and $\nu \vdash mn$, 
we have $$ Y_{\lambda[\mu]}^{\nu} \ge a_{\lambda[\mu]}^{\nu}\,\,\,\,\,.$$
\end{lemma}

\begin{proof}
Recall
$$ s_\lambda[s_\mu]=\sum_{\kappa \vdash mn}a_{\lambda[\mu]}^{\kappa}s_\kappa \,\,, $$
and comparing the coefficients of the monomial symmetric function $m_{\nu}$, 
\par we have

\begin{align*}
Y_{\lambda[\mu]}^{\nu}
&=\sum_{\kappa \vdash mn}a_{\lambda[\mu]}^{\kappa}K_{\kappa,\nu}  \\
&=a_{\lambda[\mu]}^{\nu}+
    \sum_{\kappa > \nu}a_{\lambda[\mu]}^{\kappa}K_{\kappa,\nu}
    \,\,\,\,\,,\,\,\,
    (\text{ by Corollary 2.6 (i) and } \,\, K_{\nu,\nu}=1)       \\
&\ge a_{\lambda[\mu]}^{\nu} 
    \,\,\,\,\,\,\,\,\,\,\,\,\,\,\,\,\,\,\,\,\,\,,\,\,\,
    (a_{\lambda[\mu]}^{\kappa}\ge 0 \text{ and } K_{\kappa,\nu} \ge 0).
\end{align*}

\end{proof}

%%%%%%%%%%%%%%%%%%%%%

\end{document}